\documentclass{amsart}
\usepackage[utf8]{inputenc}
\usepackage{amssymb,amsmath,enumerate,amsthm,enumitem,colonequals}
\usepackage[mathcal]{eucal}
\usepackage[margin=1in]{geometry}

\makeatletter
\@namedef{subjclassname@2020}{\textup{2020} Mathematics Subject Classification}
\makeatother

\usepackage{hyperref}
\hypersetup{
 colorlinks=true,
 linkcolor=blue,
 filecolor=magenta, 
 urlcolor=cyan,
}
\def\eqref#1{\textcolor{blue}{(\ref{#1})}}

\PassOptionsToPackage{hyphens}{url}\usepackage{hyperref}

\usepackage[pagewise]{lineno}
\let\oldequation\equation
\let\oldendequation\endequation
\renewenvironment{equation}{\linenomathNonumbers\oldequation}{\oldendequation\endlinenomath}
\expandafter\let\expandafter\oldequationstar\csname equation*\endcsname
\expandafter\let\expandafter\oldendequationstar\csname endequation*\endcsname
\renewenvironment{equation*}{\linenomathNonumbers\oldequationstar}{\oldendequationstar\endlinenomath}
\let\oldalign\align
\let\oldendalign\endalign
\renewenvironment{align}{\linenomathNonumbers\oldalign}{\oldendalign\endlinenomath}
\expandafter\let\expandafter\oldalignstar\csname align*\endcsname
\expandafter\let\expandafter\oldendalignstar\csname endalign*\endcsname
\renewenvironment{align*}{\linenomathNonumbers\oldalignstar}{\oldendalignstar\endlinenomath}

\usepackage{tikz-cd}




\newcounter{intro}
\newcounter{result}

\newtheorem{introthm}[intro]{Theorem}

\newtheorem{introcor}[intro]{Corollary}

\theoremstyle{definition}
\newtheorem{theorem}[result]{Theorem}
\newtheorem{lemma}[result]{Lemma}
\newtheorem{proposition}[result]{Proposition}
\newtheorem{corollary}[result]{Corollary}
\newtheorem{definition}[result]{Definition}
\newtheorem{example}[result]{Example}
\newtheorem{remark}[result]{Remark}

\newtheorem*{ack}{Acknowledgements}

\numberwithin{equation}{section}
\numberwithin{result}{section}

\title[Approximability, Rouquier dimension, noncommutative algebras]{Approximability and Rouquier dimension \\ for noncommutative algebras over schemes}

\author[T.~De Deyn]{Timothy De Deyn}
\address{T.~De Deyn,
School of Mathematics and Statistics,
University of Glasgow, 
Glasgow G12 8QQ,
United Kingdom}
\email{timothy.dedeyn@glasgow.ac.uk}

\author[P.~Lank]{Pat Lank}
\address{P.~Lank,
Dipartimento di Matematica “F. Enriques”, Universit\`{a} degli Studi di Milano, Via Cesare
Saldini 50, 20133 Milano, Italy}
\email{plankmathematics@gmail.com}

\author[K. ~Manali Rahul]{Kabeer Manali Rahul}
\address{K. ~Manali Rahul,
Center for Mathematics and its Applications, 
Mathematical Science Institute, Building 145, 
The Australian National University, 
Canberra, ACT 2601, Australia}
\email{kabeer.manalirahul@anu.edu.au}

\date{\today}

\keywords{noncommutative algebraic geometry, derived categories, Rouquier dimension, approximable triangulated categories}

\subjclass[2020]{14A22 (primary), 14A30, 14F08, 16E35, 18G80} 

\begin{document}
    
\begin{abstract}
    This work is concerned with approximability (\`{a} la Neeman) and Rouquier dimension for triangulated categories associated to noncommutative algebras over schemes. Amongst other things, we establish that the category of perfect complexes of a Noetherian quasi-coherent algebra over a separated Noetherian scheme is strongly generated if, and only if, there exists an affine open cover where the algebra has finite global dimension. As a consequence, we solve an open problem posed by Neeman. Further, as a first application, we study the existence of generators for Azumaya algebras.
\end{abstract}

\maketitle

\section{Introduction}
\label{sec:intro}

This work contributes to approximation and generation in triangulated categories constructed from noncommutative algebras over schemes. Thereby, it enables the study of objects within triangulated categories that arise naturally in noncommutative contexts by leveraging techniques from better understood commutative settings. A compelling consequence is the accessibility of potent machinery in noncommutative flavors of algebraic geometry, yielding interesting applications. For example, we are able to prove a problem posed by Neeman in \cite{Neeman:2021} and exceed well beyond its scope, see Theorem~\ref{introthm:coherent_alg_perf_strong_gen}.

Recently \textit{approximable triangulated categories}, introduced by Neeman in \cite{Neeman:2021} and \cite{Neeman:2021b}, have played a crucial role in resolving significant open problems in algebraic geometry, see \cite{Neeman:2021, Neeman:2021b, Neeman:2022}. Roughly speaking, approximability allows one to understand arbitrary objects in a triangulated category by approximations from simpler ones. 
Our work, which we relay below, was partly motivated by wanting to extend these approximation tools to a noncommutative set-up.

Noncommutative algebraic geometry can be viewed as both the geometric study of abstract noncommutative structures and as the application of these abstract structures to geometric spaces.
It is a broad field connecting many others such as (non)commutative algebra, algebraic geometry, representation theory and mathematical physics. 
Various different approaches to this exist and trying to give a good overview of all of these would only result in us forgetting some, hence we refrain from doing so. 

In this work we opt for a `mild' noncommutative approach.
This consists in studying quasi-coherent sheaves of noncommutative algebras over schemes. 
We refer to these as `(mild) noncommutative schemes', cf.\ Section~\ref{sec:algebras} for details. 
Examples of where these have appeared earlier are e.g.\ \cite{VandenBergh:2004flop, Yekutieli/Zhang:2006, Kuznetsov:2006, Burban/Drozd/Gavran:2017} (this list is far from exhaustive).
Throughout our work, we observe interactions between the properties of the noncommutative scheme and the underlying (commutative) scheme, which stimulates our chosen perspective. 

To set the stage, briefly recall the different notions of generators for triangulated categories, as introduced in \cite{Bondal/VandenBergh:2003}.
An object $G$ in a triangulated category $\mathcal{T}$ is called a \textit{classical generator} if any other object in $\mathcal{T}$ is obtained from $G$ by taking shifts, finite coproducts, retracts and cones.
If all object in $\mathcal{T}$ can be finitely built by taking at most $n$ cones, then $G$ is called a \textit{strong generator}. The \textit{Rouquier dimension} for $\mathcal{T}$ is the minimal such $n$ among the strong generators. 
In addition, when $\mathcal{T}$ admits small coproducts, it is often convenient to make use of arbitrary small coproducts in the definition of a strong generator, allowing this yields what is known as a  \textit{strong $\oplus$-generator}
See Section~\ref{sec:generation} for details and Examples~\ref{ex:strong_generator_examples} and \ref{ex:NCexamples} for an overview of a few existence results in the literature.

The existence of generators in a triangulated categories has many pleasant consequences. 
This includes representability theorems for cohomological functors when strong generators exist \cite{Rouquier:2008}, as well as implications for geometric properties; for instance, the bounded derived category of coherent sheaves $D^b_{\operatorname{coh}}(X)$ admitting a classical generator for a Noetherian scheme $X$ tells us openness of its regular locus \cite{Iyengar/Takahashi:2019, Dey/Lank:2024a}. Recently, it has also been demonstrated that generation in $D^b_{\operatorname{coh}}(X)$ can be utilized to provide triangulated characterizations of various singularities arising in birational geometry \cite{Lank/Venkatesh:2024,Lank/McDonald/Venkatesh:2025}.

For these reasons it is natural to study and characterise the existence of suitable generators in triangulated categories. 
Our first results build upon a notion of (noncommutative) regularity, for which we are inspired by \cite{Neeman:2021}. Then, as a first application we focus on a class of noncommutative objects called \textit{Azumaya algebras}, and make observations regarding generation for their associated derived categories of perfect complexes. 

\subsection{Regularity}
\label{sec:intro_regularity}

A conjecture by Bondal and Van den Bergh states that the category of perfect complexes $\operatorname{perf}(X)$ on a separated Noetherian scheme $X$ is strongly generated if, and only if, there exists an affine open cover of $X$ by rings of finite global dimension\footnote{This remains true in the non-Noetherian setting but the statement needs to be suitably adapted, see Theorem \ref{introthm:coherent_alg_perf_strong_gen}. The statement using a strong generator for the category of perfect complexes as opposed to a perfect strong $\oplus$-generators is false, see \cite{Stevenson:2024}.}. Recently, this has been proven in \cite[Theorem 0.5]{Neeman:2021}. Key ingredients in the proof of loc.\ cit.\ is the concept of approximability for the derived category of quasi-coherent sheaves on a quasi-compact separated scheme (\cite[Example 3.6]{Neeman:2021b}) combined with results giving control on the pushforward of perfect complexes along open immersions, i.e.\ \cite[Theorem 6.2]{Neeman:2021}.
We show noncommutative versions of these in respectively Propositions~\ref{prop:approx_for_dqcoh_coherent_algebra} and \ref{prop:compact_boundedness_noncommuative}. 

Another important consequence of approximability is Corollary \ref{cor:rouq_dim_via_approx_perf_alg}; this allows one to use the full machinery of arbitrary coproducts to construct strong $\oplus$-generators, and then conclude that they are strong generators. In particular, this allows one to make arguments about objects with bounded and coherent cohomology similar to statement that were previously only known for compact objects (i.e.\ \cite[Proposition 2.2.4]{Bondal/VandenBergh:2003}).

A noncommutative version of this characterization regarding the existence of strong generators for the category of perfect complexes was posed as a question in \cite[Remark 0.22.iv]{Neeman:2021}. This brings attention to our first main result, which exceeds far beyond the scope of the remark loc.\ cit.

\begin{introthm}\label{introthm:coherent_alg_perf_strong_gen}(see Theorem~\ref{thm:coherent_alg_perf_strong_gen})
    Let $X$ be a quasi-compact separated scheme and let $\mathcal{A}$ be a quasi-coherent $\mathcal{O}_X$-algebra. The following are equivalent:
    \begin{enumerate}
        \item There exists a perfect strong $\oplus$-generator for $D_{\operatorname{Qcoh}}(\mathcal{A})$.
        \item For each affine open $U \subseteq X$, $D_{\operatorname{Qcoh}}(\mathcal{A}|_U)$ has a perfect strong $\oplus$-generator.
        \item $ \mathcal{A}(U)$ has finite global dimension for each affine open $U\subseteq X$.
        \item $ \mathcal{A}(U_i)$ has finite global dimension for each open $U_i$ in an affine open cover $X=\cup^n_{i=1}U_i$. 
    \end{enumerate}
    Any of these equivalent conditions imply that $\operatorname{perf}(\mathcal{A})$ has finite Rouquier dimension.
\end{introthm}

\begin{introcor}\label{introcor:alg_perf_strong_gen_weak_version}(see Corollary~\ref{cor:alg_perf_strong_gen_weak_version})
     Let $X$ be a quasi-compact separated scheme and let $\mathcal{A}$ be a quasi-coherent $\mathcal{O}_X$-algebra with $\mathcal{A}(U)$ a Noetherian ring for each affine open $U\subseteq X$. The following are equivalent:
     \begin{enumerate}
        \item There exists a strong generator for $\operatorname{perf}(\mathcal{A})$.
        \item For each affine open $U \subseteq X$, $\operatorname{perf}(\mathcal{A}|_U)$ has a strong generator.
        \item $\mathcal{A}(U)$ has finite global dimension for each affine open $U \subseteq X$.
        \item $ \mathcal{A}(U_i)$ has finite global dimension for each open $U_i$ in an affine open cover $X=\cup^n_{i=1}U_i$. 
    \end{enumerate}
\end{introcor}

In particular, it follows immediately that the existence of a perfect strong $\oplus$-generator as in Theorem \ref{introthm:coherent_alg_perf_strong_gen} or a strong generator as in Corollary \ref{introcor:alg_perf_strong_gen_weak_version} can be checked affine locally (see Remark~\ref{rmk:stong_gen_to_strong_oplus_for_perf}).
These results, and those in \S 4, are the cornerstone to many interesting applications.
A first of which can be found in Section~\ref{sec:intro_azumaya} below, and more appear in \cite{DeDeyn/Lank/ManaliRahul:2024c}.  
In addition, it ties together several natural notions of regularity. Consequently, it is natural to wonder whether an adequate theory of regularity can be established in our noncommutative set-up.
While we do not explore this concept here, it will be addressed in our forthcoming work \cite{DeDeyn/Lank/ManaliRahul:2024b}. 

\subsection{Azumaya algebras}
\label{sec:intro_azumaya}

Here we discuss a first applications of Theorem~\ref{introthm:coherent_alg_perf_strong_gen} and approximability in the context of Azumaya algebras. Recall that an \textit{Azumaya algebra} over a scheme is a coherent $\mathcal{O}_X$-algebra that is \'{e}tale locally a matrix algebra. These algebras hold significant interest across various fields \cite{Artin:1969, Manin:1971, Saltman:1978} and they are, up to a suitable equivalence, the elements of the scheme's Brauer group.

The following provides a noncommutative characterisation of regularity for a Noetherian scheme by imposing constraints on triangulated categories arising from elements of its Brauer group. The following is well-known to experts in some form, see e.g.\ \cite{Kuznetsov:2006}, however we could not locate a reference in the literature in this generality. In any case, we want to emphasise that we give a proof of this using Theorem~\ref{introthm:coherent_alg_perf_strong_gen} (and hence, approximable triangulated categories).

\begin{introcor}\label{introcor:azumaya_detects_regular}(see Corollary~\ref{cor:azumaya_detects_regular})
    If $X$ is a separated Noetherian scheme, then the following are equivalent:
    \begin{enumerate}
        \item $\operatorname{perf}(X)$ has finite Rouquier dimension, i.e.\ $X$ is regular of finite Krull dimension.
        \item $\operatorname{perf}(\mathcal{A})$ has finite Rouquier dimension for all Azumaya algebras $\mathcal{A}$ over $X$.
        \item $\operatorname{perf}(\mathcal{A})$ has finite Rouquier dimension for some Azumaya algebra $\mathcal{A}$ over $X$.
    \end{enumerate}
\end{introcor}

Corollary~\ref{introcor:azumaya_detects_regular} extends \cite[Lemma 10.19]{Kuznetsov:2006}, where a similar observation was made for Azumaya algebras over embeddable varieties over a field.
In particular, it extends loc.\ cit.\ to settings of mixed characteristic.
Further applications, such as versions of descent, can be found in \cite{DeDeyn/Lank/ManaliRahul:2024c}.

\subsection{Notation}
\label{sec:intro_notation}

Let $X$ be a scheme and $\mathcal{A}$ be a quasi-coherent $\mathcal{O}_X$-algebra. We will work with the following triangulated categories:
\begin{enumerate}
    \item $D(\mathcal A):= D(\operatorname{Mod}(\mathcal{A}))$ is the derived category of complexes of (right) $\mathcal{A}$-modules,
    \item $D_{\operatorname{Qcoh}}(\mathcal{A})$ is the (strictly full) subcategory of $D(\mathcal{A})$ of objects with quasi-coherent cohomology,
    \item $D^b_{\operatorname{coh}}(\mathcal{A})$ is the (strictly full) subcategory of $D_{\operatorname{Qcoh}}(\mathcal{A})$ consisting of objects with bounded and coherent cohomology,
    \item $\operatorname{perf}(\mathcal{A})$ is the (strictly full) subcategory of $D_{\operatorname{Qcoh}}(\mathcal{A})$ consisting of perfect complexes.
\end{enumerate}
Of course, when $\mathcal{A}=\mathcal{O}_X$ we write $D(X)$ instead of $D(\mathcal{O}_X)$ etc.

\begin{ack}
    Timothy De Deyn was supported under the ERC Consolidator Grant 101001227 (MMiMMa). Pat Lank was partly supported by National
    Science Foundation under Grant No.\ DMS-2302263. Kabeer Manali Rahul was supported under the ERC Advanced Grant 101095900-TriCatApp, is a recipient of an Australian Government Research Training Program Scholarship, and would like to thank Universit\`{a} degli Studi di Milano for their hospitality during his stay there.
\end{ack}
    
\section{Generation \& approximation}
\label{sec:generation_approximation}

This section is a quick and dirty exposition regarding generation and approximation in triangulated categories, and so the expert can safely skip it. Throughout let $\mathcal{T}$ be a triangulated category with shift functor $[1]\colon \mathcal{T} \to \mathcal{T}$.

\subsection{Generation}
\label{sec:generation}

The primary sources of references for the contents below are \cite{Bondal/VandenBergh:2003, Rouquier:2008, Avramov/Buchweitz/Iyengar/Miller:2010, Neeman:2021a, Neeman:2021, Neeman:2021b}.

\begin{definition}(\cite{Bondal/VandenBergh:2003, Neeman:2021})\label{def:thick_subcategories}
    Let $\mathcal{S}$ be a subcategory of $\mathcal{T}$.
    \begin{enumerate} 
        \item $\mathcal{S}$ is said to be \textbf{thick} if it is a triangulated subcategory which is closed under direct summands.
        \item $\langle \mathcal{S} \rangle$ denotes the smallest thick subcategory containing $\mathcal{S}$ in $\mathcal{T}$; moreover, $\langle \mathcal{S}\rangle$ will be written as $\langle G\rangle$ whenever $\mathcal{S}$ consists of a single object $G$.
        \item $\operatorname{add}(\mathcal{S})$ is the smallest strictly full subcategory of $\mathcal{T}$ containing $\mathcal{S}$ which is closed under shifts, finite coproducts, and direct summands.
        \item $\langle \mathcal{S} \rangle_0$ consists of all objects isomorphic to the zero object and $\langle \mathcal{S} \rangle_1 := \operatorname{add}(\mathcal{S})$
        \item $\langle \mathcal{S} \rangle_n := \operatorname{add} \{ \operatorname{cone}(\phi) : \phi \in \operatorname{Hom}_{\mathcal{T}} (\langle \mathcal{S} \rangle_{n-1}, \langle \mathcal{S} \rangle_1) \}$; again, this will be denoted by $\langle G\rangle_n$ whenever $\mathcal{S}$ consists of a single object $G$.
    \end{enumerate}
    Assume that $\mathcal{T}$ admits small coproducts, in addition to the above we then define
    \begin{enumerate}[resume]
        \item $\operatorname{Add}(\mathcal{S})$ is the smallest strictly full subcategory containing $\mathcal{S}$ which is closed under shifts, small coproducts, and direct summands.
        \item $\overline{\langle \mathcal{S} \rangle_0}$  consists of all objects isomorphic to the zero object and $\overline{\langle \mathcal{S} \rangle_1} := \operatorname{Add}(\mathcal{S})$
        \item $\overline{\langle \mathcal{S} \rangle_n}:= \operatorname{Add} \{ \operatorname{cone}(\phi) : \phi \in \operatorname{Hom}_{\mathcal{T}} (\overline{\langle \mathcal{S} \rangle_{n-1}}, \overline{\langle \mathcal{S} \rangle_1}) \}$;  again, this will be denoted by $\overline{\langle G\rangle}_n$ whenever $\mathcal{S}$ consists of a single object $G$.
    \end{enumerate}
\end{definition}

\begin{remark}
    It is important to note that notation slightly differs throughout the literature. For instance, we closely follow \cite{Bondal/VandenBergh:2003}, but we also allow $\operatorname{add}$ and $\operatorname{Add}$ to be closed under direct summands, which is again slightly different from \cite{Neeman:2021}. The primary `difference' between our definitions, \cite{Bondal/VandenBergh:2003} and \cite{Neeman:2021} is whether or not one allows $\operatorname{add}$ and $\operatorname{Add}$ to be closed under direct summands and/or shifts; see also \cite[Remark 1.2]{Neeman:2021}. 
\end{remark}

\begin{remark}
    In the notation of Definition~\ref{def:thick_subcategories}, there exists an exhaustive filtration for $\langle \mathcal{S} \rangle$:
    \begin{displaymath}
        \langle \mathcal{S} \rangle_0 \subseteq \langle \mathcal{S} \rangle_1 \subseteq \dots \subseteq \bigcup^{\infty}_{n=0} \langle \mathcal{S} \rangle_n = \langle \mathcal{S} \rangle.
    \end{displaymath}
\end{remark}

\begin{definition}(\cite{Rouquier:2008,Avramov/Buchweitz/Iyengar/Miller:2010})\label{def:strong_generators}
    Let $E,G$ be objects of $\mathcal{T}$, and $\mathcal{S}$ a subcategory of $\mathcal{T}$.
    \begin{enumerate}
        \item $\operatorname{level}_\mathcal{T}^{\mathcal{S}} (E)$, called the \textbf{level} of $E$ with respect to $\mathcal{S}$, is the minimal non-negative integer $n$ such that $E$ is in $\langle \mathcal{S} \rangle_n$; if $\mathcal{S}$ consists of a single object $G$, we write $\operatorname{level}_\mathcal{T}^{\mathcal{S}} (E)$ as $\operatorname{level}_\mathcal{T}^G (E)$.
        \item We say  $E$ is \textbf{finitely built} by $\mathcal{S}$ if $E$ if its level with respect to $\mathcal{S}$ is finite.
        \item $G$ is called a \textbf{classical generator} for $\mathcal{T}$ if $\langle G \rangle = \mathcal{T}$.
        \item $G$ is called a \textbf{strong generator} for $\mathcal{T}$ if $\langle G \rangle_{n+1} = \mathcal{T}$ for some $n\geq 0$, and the \textbf{generation time} of $G$ is the minimal $n$ such that for all $E$ in $ \mathcal{T}$ one has $\operatorname{level}_\mathcal{T}^G (E)\leq n+1$.
        \item $\dim \mathcal{T}$, called the \textbf{Rouquier dimension} of $\mathcal{T}$, is the smallest integer $d$ such that $\langle G \rangle_{d+1} = \mathcal{T}$ for some object $G$ in $ \mathcal{T}$, i.e.\ it is the minimal generation time of the strong generators.
    \end{enumerate}
    If, in addition, $\mathcal{T}$ has small coproducts,
    \begin{enumerate}[resume]
        \item $G$ is called a \textbf{strong $\oplus$-generator} for $\mathcal{T}$ if $\overline{\langle G \rangle}_{n+1} = \mathcal{T}$ for some $n\geq 0$.
    \end{enumerate}
\end{definition}

\begin{example}\label{ex:strong_generator_examples}
    The following classes of schemes $X$ have the property that $D^b_{\operatorname{coh}}(X)$ admits a strong or classical generator:
    \begin{enumerate}
        \item If $X$ is a quasi-affine Noetherian regular scheme of finite Krull dimension, then $\langle \mathcal{O}_X \rangle_{\dim X + 1} = \allowbreak D^b_{\operatorname{coh}}(X)$, see \cite[Corollary 5]{Olander:2023}.
        \item For any smooth quasi-projective variety $X$ over a field with very ample line bundle $\mathcal{L}$, one has $\langle \oplus^{\dim X}_{i=0} \mathcal{L}^{\otimes i} \rangle_{2 \dim X  +1 } = D^b_{\operatorname{coh}}(X)$, see \cite[Proposition 7.9]{Rouquier:2008}.
        \item Let $X$ be a separated Noetherian scheme of prime characteristic. Denote the $e$-th iterate of the Frobenius morphism of $X$ by $F^e \colon X \to X$. Assume $X$ is $F$-finite, i.e. $F\colon X \to X$ is a finite morphism. If $P$ is a classical generator for $\operatorname{perf}(X)$, then $F_\ast^e P$ is a strong generator for  $D^b_{\operatorname{coh}}(X)$ with $e \gg 0$, see \cite[Corollary 3.9]{BILMP:2023}. This is applicable to quasi-projective varieties over a perfect field of prime characteristic.
        \item Any separated quasi-excellent Noetherian scheme of finite Krull dimension admits a strong generator, see \cite[Main Theorem]{Aoki:2021}. For related work in the affine setting, see \cite{Iyengar/Takahashi:2019, Dey/Lank/Takahashi:2023}.
        \item If $X$ is a Noetherian scheme such that every closed integral subscheme of $X$ has open regular locus, then $D^b_{\operatorname{coh}}(X)$ admits a classical generator, see \cite[Theorem 1.1]{Dey/Lank:2024a}.
    \end{enumerate}
\end{example}

\begin{definition}(\cite{Neeman:2021})
    Let $\mathcal{S}$ be a subcategory of $\mathcal{T}$. Suppose $a,b$ are in $\mathbb{Z} \cup \{\pm \infty \}$. Consider the following additive subcategories of $\mathcal{T}$:
    \begin{enumerate}
        \item $\mathcal{S}(a,b) = \{ S[i] : -b < i < -a, S \in \mathcal{S} \}$.
        \item $\operatorname{coprod}_0 (\mathcal{S}(a,b))$  consists of all objects isomorphic to the zero object and $\operatorname{coprod}_1 (\mathcal{S}(a,b))$ is the smallest strictly full subcategory closed under finite coproducts of objects in $\mathcal{S}(a,b)$.
        \item $\operatorname{coprod}_n (\mathcal{S}(a,b))$ is the smallest strictly full subcategory containing objects $E$ such that there is a distinguished triangle in $\mathcal{T}$:
        \begin{displaymath}
            A \to E \to B \to A [1],
        \end{displaymath}
        where $A$ is in $\operatorname{coprod}_1 (\mathcal{S}(a,b))$ and $B$ is in $\operatorname{coprod}_{n-1} (\mathcal{S}(a,b))$.
        \item $\langle \mathcal{S} \rangle_n^{(a,b)}$ the smallest strictly full subcategory closed under direct summands which contains the category $\operatorname{coprod}_n (\mathcal{S}(a,b))$.
    \end{enumerate}
    Additionally, if $\mathcal{T}$ admits small coproducts, we also define
    \begin{enumerate}
        \item $\operatorname{Coprod}_0 (\mathcal{S}(a,b))$  consists of all objects isomorphic to the zero object and $\operatorname{Coprod}_1 (\mathcal{S}(a,b))$ is the smallest strictly full subcategory closed under small coproducts of objects in $\mathcal{S}(a,b)$.
        \item $\operatorname{Coprod}_n (\mathcal{S}(a,b))$ is the smallest strictly full subcategory containing objects $E$ such that there is a distinguished triangle in $\mathcal{T}$:
        \begin{displaymath}
            A \to E \to B \to A [1],
        \end{displaymath}
        where $A$ is in $\operatorname{Coprod}_1 (\mathcal{S}(a,b))$ and $B$ is in $\operatorname{Coprod}_{n-1} (\mathcal{S}(a,b))$.
        \item $\overline{\langle \mathcal{S} \rangle}_n^{(a,b)}$ is the smallest strictly full subcategory closed under direct summands which contains the category $\operatorname{Coprod}_n (\mathcal{S}(a,b))$.
    \end{enumerate}
    We permit various intervals such as $[a,b],[a,b),(a,b]$ that are defined analogously.
\end{definition}

\begin{definition}
    Let $\mathcal{T}$ be a triangulated category with small coproducts.
    An object $G$ in $\mathcal{T}$ is called a \textbf{compact generator} for $\mathcal{T}$ if $G$ is compact and one has that if for any object $F \in \mathcal{T}$, $\operatorname{Hom}(G,F[n]) = 0$ for all integers $n$ implies that $F = 0 $. This is equivalent to $\mathcal{T}$ being the localizing category generated by $G$. The collection of compact objects in $\mathcal{T}$ is denoted $\mathcal{T}^c$.
\end{definition}

\begin{remark}
    If $\mathcal{T}$ is compactly generated, an object is a classical generator for $\mathcal{T}^c$ if and only if it is a compact generator for $\mathcal{T}$, see e.g.\ \cite[\href{https://stacks.math.columbia.edu/tag/09SR}{Tag 09SR}]{StacksProject}.
\end{remark}

\subsection{Approximation}
\label{sec:approximation}

The primary sources of references for what follows below are \cite{Keller/Vossieck:1988, Beilinson/Berstein/Deligne/Gabber:2018, Neeman:2022}. Assume $\mathcal{T}$ admits all small coproducts.

\begin{definition}(\cite{Beilinson/Berstein/Deligne/Gabber:2018})\label{def:t_structure}
    A pair of strictly full subcategories $(\mathcal{T}^{\leq 0},\mathcal{T}^{\geq 0})$ in $\mathcal{T}$ is called a \textbf{$t$-structure} if the following conditions are satisfied:
    \begin{enumerate}
        \item $\operatorname{Hom}(A,B)=0$ for all $A$ in $\mathcal{T}^{\leq 0}$ and $B$ in $\mathcal{T}^{\geq 0}[-1]$,
        \item $\mathcal{T}^{\leq 0}[1]\subseteq \mathcal{T}^{\leq 0}$ and $\mathcal{T}^{\geq 0}[-1]\subseteq \mathcal{T}^{\geq 0}$,
        \item for any object $E$ in $\mathcal{T}$, there exists a distinguished triangle
        \begin{displaymath}
            A \to E \to B \to A[1]
        \end{displaymath}
        with $A$ in $\mathcal{T}^{\leq 0}$ and $B$ in $\mathcal{T}^{\geq 0}[-1]$.
    \end{enumerate}
\end{definition}

If $(\mathcal{T}^{\leq 0}, \mathcal{T}^{\geq 0})$ is a $t$-structure, then for each integer $n$ one has that the pair $(\mathcal{T}^{\leq n}, \mathcal{T}^{\geq n})$ is also a $t$-structure on $\mathcal{T}$ where $\mathcal{T}^{\leq n}:= \mathcal{T}^{\leq 0}[-n]$ and $\mathcal{T}^{\geq n}:= \mathcal{T}^{\geq 0}[-n]$. We will use the following notation
    \begin{displaymath}
        \mathcal{T}^{-}:= \bigcup_{n=0}^\infty \mathcal{T}^{\leq n},\quad \mathcal{T}^{+} := \bigcup^\infty_{n=0} \mathcal{T}^{\geq -n},\quad \mathcal{T}^b := \mathcal{T}^{-} \cap \mathcal{T}^{+}.
    \end{displaymath}

\begin{example}\label{ex:standard_t-strct_DQcoh}
    Let $X$ be a scheme. Consider the following subcategories:
    \begin{enumerate}
        \item $D^{\leq 0}_{\operatorname{Qcoh}}(X)$ consisting of objects $E$ in $D_{\operatorname{Qcoh}}(X)$ with vanishing strictly positive cohomology, i.e.\ $\mathcal{H}^j (E) = 0$ for $j >0$,
        \item $D^{\geq 0}_{\operatorname{Qcoh}}(X)$ consisting of objects $E$ in $D_{\operatorname{Qcoh}}(X)$ with vanishing strictly negative cohomology, i.e.\  $\mathcal{H}^j (E) = 0$ for $j <0$.
    \end{enumerate}
    The pair $(D^{\leq 0}_{\operatorname{Qcoh}}(X), D^{\geq 0}_{\operatorname{Qcoh}}(X))$ gives a $t$-structure on $D_{\operatorname{Qcoh}}(X)$, called the \textbf{standard $t$-structure}
    More generally, one can, of course, play the same game for more general derived categories. 
\end{example}

\begin{definition}\label{def:approx}(\cite{Neeman:2021a})
    A triangulated category $\mathcal{T}$ with small coproducts is said to be \textbf{approximable} if there exists a compact generator $G$ for $\mathcal{T}$, a $t$-structure $(\mathcal{T}^{\leq 0}, \mathcal{T}^{\geq 0})$ on $\mathcal{T}$, and a positive integer $A$ such that the following conditions hold:
    \begin{enumerate}
        \item $G[A]$ is in $\mathcal{T}^{\leq 0}$,
        \item $\operatorname{Hom}(G[-A],E)=0$ for all $E$ in $\mathcal{T}^{\leq 0}$,
        \item for each object $F$ in $\mathcal{T}^{\leq 0}$ there exists a distinguished triangle
        \begin{displaymath}
            E \to F \to D \to E[1],
        \end{displaymath}
        where $D$ is in $\mathcal{T}^{\leq -1}$ and $E$ is in $\overline{\langle G \rangle}_A^{[-A,A]}$.
    \end{enumerate}
\end{definition}

\begin{example}
    If $X$ is a quasi-compact quasi-separated scheme, then there exists a compact generator for $D_{\operatorname{Qcoh}}(X)$, see \cite[Theorem 3.1.1]{Bondal/VandenBergh:2003}. 
    Moreover, if $X$ is separated, then $D_{\operatorname{Qcoh}}(X)$ is approximable, see \cite[Example 3.6]{Neeman:2021b}.
\end{example}

\begin{definition}(\cite{Neeman:2018})\label{def:T^-_c}
    Let $(\mathcal{T}^{\leq 0},\mathcal{T}^{\geq 0})$ be a $t$-structure on $\mathcal{T}$. 
    \begin{enumerate}
        \item Define $\mathcal{T}^{-}_c$ as the collection of objects $E$ in $\mathcal{T}$ such that for each non-negative integer $n$ there exists a distinguished triangle
        \begin{displaymath}
            P \to E \to F \to P[1],
        \end{displaymath}
        where $P$ is in $\mathcal{T}^c$ and $F$ is in $\mathcal{T}^{\leq -n}$.
        \item Define $\mathcal{T}^b_c$ as the intersection of $\mathcal{T}^{-}_c$ and $\mathcal{T}^b$.
    \end{enumerate}
\end{definition}

The following statement relates the two constructions $\langle - \rangle_n$ and $\overline{\langle - \rangle}_n$ under specific conditions. It was initially noted for schemes in \cite[Remark 2.19]{Lank:2024}, drawing on \cite[Lemma 2.4 and Lemma 2.6]{Neeman:2021}, but has been significantly strengthened in \cite[Lemma 2.26]{Biswas/Chen/Rahul/Parker/Zheng:2023}.

\begin{lemma}\label{lem:big_coprod_intersection_big_thick}
    Let $\mathcal{T}$ be a triangulated category with small coproducts and a $t$-structure such that $\mathcal{T}^{\geq 0}$ is closed under small coproducts\footnote{Note that this is true for any compactly generated t-structure.}. Fix an arbitrary integer $n$ and positive integer $N$. Choose $G$ in $\mathcal{T}^b$. 
    \begin{enumerate}
        \item There exists an integer $M$ such that $\mathcal{T}^{\geq n} \cap \overline{\langle G \rangle}_N$ is contained in $\overline{ \langle G \rangle }^{[n-M,\infty)}_N$
        \item Suppose $\mathcal{S}$ is a triangulated subcategory of $\mathcal{T}$ contained in $\mathcal{T}^-_c$. Assume $G$ is in $\mathcal{S} \cap \mathcal{T}^b$. Any map $f\colon E \to F$ with $E$ an object of $\mathcal{S}$ and $F$ an object of $\overline{ \langle G \rangle}^{[M, \infty)}_n $ factors through an object $F^{\prime}$ in $\langle G \rangle^{[M,\infty)}_n$.
        \item If $G$ is an object of $\mathcal{T}^b_c$, then $\mathcal{T}^b_c \cap \overline{\langle G \rangle}_N= \langle G \rangle_N$.
    \end{enumerate}
\end{lemma}

\begin{proof}
     This is readily deduced from \cite[Lemma 2.26]{Biswas/Chen/Rahul/Parker/Zheng:2023} and the fact that any object in $\overline{\langle \mathcal{S}\rangle}_n^{(a,b)}$ is a direct summand of an object in $\operatorname{Coprod}_n (\mathcal{S}(a, b))$. 
\end{proof}

\section{Algebras over schemes}
\label{sec:algebras}

This section discusses the main protagonist of our work, which are schemes equipped with noncommutative structure sheaves.

\begin{definition}\label{def:nc_scheme}
\hfill
    \begin{enumerate}
        \item A \textbf{(mild) noncommutative scheme} consists of a pair $(X,\mathcal A)$ where $X$ is a scheme and $\mathcal A$ is a quasi-coherent sheaf of $\mathcal O_X$-algebras.
        \item A \textbf{morphism of (mild) noncommutative schemes} $(f, \varphi)\colon(Y, \mathcal B)\to (X,\mathcal A)$ consists of a morphism of schemes $f\colon Y\to X$ and a morphism\footnote{This is equivalent to a morphism $f^\ast\mathcal A\to\mathcal B$ of sheaves of algebras over $f^\ast\mathcal O_{X}\cong\mathcal O_{Y}$.} $\varphi\colon f^{-1}\mathcal A\to\mathcal B$ of sheaves of $f^{-1}\mathcal O_{X}$-algebras. 
    \end{enumerate}
\end{definition}

We will omit the word `mild', suppress the $\varphi$ in the notation, and write `nc.\ scheme' for the sake of brevity in Definition~\ref{def:nc_scheme}. 

Consider a nc.\ scheme $(X,\mathcal{A})$. Unless otherwise specified, when $\mathcal{P}$ is a topological property of schemes (resp.\ morphisms of schemes) we say that $(X,\mathcal A)$ (resp.\ $(f,\varphi)$) has $\mathcal{P}$ if $X$ (resp.\ $f$) has $\mathcal{P}$. Adjectives referring to algebraic properties are more subtle.

\begin{definition}
    A nc.\ scheme $(X,\mathcal A)$ is called 
   \begin{enumerate}
        \item \textbf{affine} if $X=\operatorname{Spec}(R)$ for some commutative ring $R$,
        \item \textbf{coherent} if $\mathcal A$ is coherent as an $\mathcal{O}_X$-module,
        \item \textbf{integral} if its underlying scheme $X$ is integral,
        \item \textbf{Noetherian} if its underlying scheme $X$ is Noetherian.
   \end{enumerate} 
\end{definition}

\begin{remark}
    The nc.\ scheme $(X,\mathcal{A})$ being both coherent and Noetherian implies that $\mathcal A$ is Noetherian in the sense that it satisfies the ascending chain condition on quasi-coherent ideals (i.e.\ it is a Noetherian object in the category of quasi-coherent $\mathcal{A}$-modules) or equivalently the sections over affines are Noetherian rings.
\end{remark}

In fact, coherent Noetherian affine nc.\ schemes already appear in the (more algebraic) literature by another name.
Recall that a \textbf{Noether algebra} is a Noetherian (potentially noncommutative) ring which is finite (as module) over its center.
The coherent Noetherian affine nc.\ schemes are exactly the Noether algebras.
For this reason we adopt the following terminology.

\begin{definition}
    A noncommutative scheme $(X,\mathcal{A})$ is called a \textbf{Noether (noncommutative) scheme} if it is coherent and Noetherian.
\end{definition}

Clearly, one can view schemes as nc.\ schemes by simply equipping them with their structure sheaf, i.e.\ by considering $(X,\mathcal O_X)$. In this way we can identify the category of schemes as a full subcategory of the category of  nc.\ schemes. We do this throughout, simply writing $X$ instead of $(X,\mathcal O_X)$.

\begin{example}
    Recall that a coherent sheaf of algebras $\mathcal A$ over a scheme $X$ is called an \textbf{Azumaya algebra} if it is \'etale (or fppf) locally a matrix algebra. For more information and alternative characterisations see e.g.\ \cite{Milne:1980}. 
    We will say a nc.\ scheme $(X,\mathcal{A})$ is an \textbf{Azumaya scheme} if $\mathcal{A}$ is an Azumaya algebra over a Noetherian scheme $X$ (so we assume Noetherian). See \cite[Appendix D]{Kuznetsov:2006} for background on derived categories arising from Azumaya algebras over varieties.
\end{example}

\begin{example}
    A more general class of examples are \textbf{orders} over schemes, by which we mean torsion-free coherent algebras that are generically Azumaya.
    Examples are the endomorphism sheaves of reflexive sheaves over varieties, e.g.\ the noncommutative crepant resolutions of Van den Bergh \cite{VandenBergh:2004flop, VandenBergh:2004nccr} or the (maximal) modifying algebras of Iyama and Wemyss \cite{Iyama/Wemyss:2014a,Iyama/Wemyss:2014b,Wemyss:2018}.
    See also the end of the introduction of \cite{Burban/Drozd:2022} for some historical perspective concerning orders over curves or \cite{Chan/Ingalls:2005} for work on orders over surfaces.
\end{example}

Let $(X,\mathcal A)$ be a  nc.\ scheme.
One can define the category of (right) $\mathcal A$-module sheaves, and the (quasi-)coherent versions in the usual fashion (the definitions make sense over any ringed space, irrespective of whether the structure sheaf is commutative).
We will denote the respective categories by $\operatorname{Mod}(X, \mathcal A)$, $\operatorname{Qcoh}(X, \mathcal A)$ and $\operatorname{coh}(X, \mathcal A)$, or simply $\operatorname{Mod}(\mathcal A)$, $\operatorname{Qcoh}(\mathcal A )$ and $\operatorname{coh}(\mathcal A)$ when $X$ is clear from context.
It is worthwhile to note that since $\mathcal A$ is quasi-coherent, one can check quasi-coherence of an $\mathcal A$-module on the underlying $\mathcal O_X$-module level. The same is true for coherence when $\mathcal A$ is coherent.
In addition, as usual, a morphism $f\colon(Y, \mathcal B)\to (X,\mathcal A)$ induces the pullback-pushforward adjunction between module categories
\begin{displaymath}
        \begin{tikzcd}[sep=2.5em]
            \operatorname{Mod}(\mathcal{B}) \\ \operatorname{Mod}(\mathcal{A})\rlap{ .}
            \arrow[""{name=0, anchor=center, inner sep=0}, "{f_{\ast}}", bend left=45, from=1-1, to=2-1]
            \arrow[""{name=1, anchor=center, inner sep=0}, "{f^{\ast}}", bend left=45, from=2-1, to=1-1]
            \arrow["\dashv"{anchor=center}, draw=none, from=1, to=0]
        \end{tikzcd}
\end{displaymath}
With the usual hypotheses these restrict to the (quasi-)coherent categories.

\begin{example}
    Let $(X,\mathcal A)$ be a  nc.\ scheme. The sheaf $\mathcal A$ being a sheaf
    of $\mathcal O_X$-algebras comes equipped with a morphism $\mathcal O_X \to
    \mathcal A$. This in turn induces a morphism $\pi\colon (X, \mathcal A)\to
    X$ which we refer to as the \textbf{structure morphism}. The induced
    pullback-pushforward adjunction is just the usual extension and restriction
    of scalars in this case 
    \begin{displaymath}
        \begin{tikzcd}[sep=2.5em]
            {\operatorname{Mod}(\mathcal A)} \\ {\operatorname{Mod}(X)}\rlap{ .}
            \arrow[""{name=0, anchor=center, inner sep=0}, "{\pi_{\ast}=(-)|_{\mathcal O_X}}", bend left=45, from=1-1, to=2-1]
            \arrow[""{name=1, anchor=center, inner sep=0}, "{-\otimes_{\mathcal O_X}\mathcal A=\pi^{\ast}}", bend left=45, from=2-1, to=1-1]
            \arrow["\dashv"{anchor=center}, draw=none, from=1, to=0]
        \end{tikzcd}
    \end{displaymath}
    Note that $\pi_\ast$ is an exact functor (e.g.\ note that exactness does not care about the module actions).
\end{example}

\begin{example}
    Let $(X,\mathcal{A})$ be a nc.\ scheme and $f\colon Y\to X$ a morphism of schemes. 
    This induces a morphism of nc.\ schemes $(Y, f^*\mathcal{A})\to (X,\mathcal{A})$ by taking $\varphi\colon f^{-1}\mathcal A\to f^*\mathcal{A}=f^{-1}\mathcal A\otimes_{f^{-1}\mathcal O_X}\mathcal O_Y$ to be the obvious map.
    This is compatible with the structure morphisms, and so by abuse of notation we will denote the resulting morphism of nc.\ schemes again by $f$.
\end{example}

We let $D(X, \mathcal A):= D(\operatorname{Mod}(X,\mathcal{A}))$ denote the derived category of $\mathcal A$-modules. In addition, we let $D_{\operatorname{Qcoh}}(X, \mathcal{A})$ and $ D^b_{\operatorname{coh}}(X, \mathcal{A})$  denote respectively the strictly full subcategories of $D(\operatorname{Mod}(X,\mathcal{A}))$ of those objects with quasi-coherent and with bounded and coherent cohomology\footnote{It is worthwhile to note that for $(X,\mathcal A)$ quasi-compact and separated the natural functor $D(\operatorname{Qcoh}(X,\mathcal A))\to D_{\operatorname{Qcoh}}(X,\mathcal A)$ is an equivalence. Whilst for $(X,\mathcal A)$ Noetherian and coherent the natural functor $D^-(\operatorname{coh}(X,\mathcal A))\to D^-_{\operatorname{coh}}(\operatorname{Qcoh}(X,\mathcal A))$ is an equivalence. The same proofs as in the commutative case work, see e.g.\ \cite[Proposition 1.3]{Alonso/Jeremias/Lipman:1997} and \cite[\href{https://stacks.math.columbia.edu/tag/0FDA}{Tag 0FDA}]{StacksProject}.}.
Moreover, for $Z$ a closed subset of $X$, $D_{Z}(X,\mathcal{A})$, $D_{\operatorname{Qcoh},Z}(X,\mathcal{A})$ and $D^b_{\operatorname{coh},Z}(X,\mathcal{A})$ denote the strictly full triangulated subcategories consisting of those complexes supported on $Z$.
Again when clear from context we omit the $X$ from notation, simply writing $D(\mathcal A)$, $D_{\operatorname{Qcoh}}(\mathcal{A})$, etc. 
Moreover, note that, as in Example \ref{ex:standard_t-strct_DQcoh}, $D_{\operatorname{Qcoh}}(X,\mathcal{A})$ enjoys the standard $t$-structure given by $D_{\operatorname{Qcoh}}^{\leq 0} (X,\mathcal{A})$ consisting of those objects with vanishing strictly positive cohomology and $D_{\operatorname{Qcoh}}^{\geq 0} (X, \mathcal{A})$ consisting of those objects with vanishing strictly negative cohomology.

As for schemes various derived functors exist and these can be computed as usual
using $K$-injective and $K$-flat resolutions following Spaltenstein
\cite{Spaltenstein:1988}. In particular, any morphism $f$ of nc.\ schemes induces
the derived pullback-pushforward adjunction $\mathbb Lf^\ast \dashv \mathbb
Rf_\ast$. Many of the usual formulas between various derived functors for
schemes hold in the noncommutative set-up, see e.g.\ \cite[Appendix
D]{Kuznetsov:2006} or \cite[Proposition 3.12]{Burban/Drozd/Gavran:2017}. For
future reference, let us note the following lemmas.

\begin{lemma}\label{lem:proper_to_pullback_dbcoh}
    Let $f \colon (Y,\mathcal B) \to (X, \mathcal A)$ be a proper morphism of Noether  nc.\ schemes. 
    The derived pushforward restricts to give an exact functor $\mathbb{R}f_\ast \colon D^b_{\operatorname{coh}}(\mathcal{B}) \to D^b_{\operatorname{coh}}( \mathcal{A})$.
\end{lemma}

\begin{proof}
    As the  nc.\ schemes are coherent, one can check coherence on the underlying $\mathcal{O}_X$- and $\mathcal{O}_Y$-level.
    Hence, the statement follows immediately from the corresponding statement for schemes, see e.g.\ \cite[Corollary 1.4.12 and Theorem 3.2.1]{EGAIII1:1961}, using the fact that the following square (2-)commutes.
	\begin{displaymath}
        \begin{tikzcd}[column sep=4em]
            {D_{\operatorname{Qcoh}}(\mathcal{B})} & {D_{\operatorname{Qcoh}}(Y)} \\
            {D_{\operatorname{Qcoh}}(\mathcal{A})} & {D_{\operatorname{Qcoh}}(X)}\rlap{ ,}
            \arrow["{\mathbb R f_\ast}"', from=1-1, to=2-1]
            \arrow["{\pi_{X,\ast}}", from=1-1, to=1-2]
            \arrow["{\pi_{Y,\ast}}", from=2-1, to=2-2]
            \arrow["{\mathbb R f_\ast}", from=1-2, to=2-2]
        \end{tikzcd}
    \end{displaymath}
    where $\pi_?$ denotes the corresponding structure morphism and we also used $f$ to denote the morphism of the underlying schemes.
\end{proof}

\begin{lemma}\label{lem:pushforward_cohomologically_bounded}
    Let $(X,\mathcal{A})$ be a quasi-compact separated nc.\ scheme and $i \colon U \to X$ be an open immersion.
    Then $\mathbb{R}i_*$ has finite cohomological dimension, i.e.\ there exists $t \geq 0$ such that $\mathbb{R}i_\ast D_{\operatorname{Qcoh}}^{\leq 0}(i^\ast \mathcal{A})$ is contained in $D_{\operatorname{Qcoh}}^{\leq t}(\mathcal{A})$.
\end{lemma}

\begin{proof}
    As in the proof of Lemma \ref{lem:proper_to_pullback_dbcoh} this follows from the corresponding statement for schemes, see e.g.\ \cite[Corollary 1.4.12]{EGAIII1:1961}.
\end{proof}

To finish the barrage of definitions let us give an explicit description of the
compact objects in the category $D_{\operatorname{Qcoh}}(\mathcal{A})$ when
$X$ is quasi-compact and quasi-separated. As in the commutative case these are
exactly the perfect complexes. 

\begin{definition}
    A complex in $D_{\operatorname{Qcoh}}(\mathcal{A})$ is called
    \textbf{perfect} if it is (Zariski) locally quasi-isomorphic to a bounded complex
    consisting of $\mathcal A$-modules which are themselves direct summands of
    finite free $\mathcal A$-modules (i.e.\ `finitely generated projective').
    Let $\operatorname{perf}(X,\mathcal A)$ denote the strictly full subcategory
    of $D_{\operatorname{Qcoh}}(\mathcal{A})$ consisting of perfect complexes
    and $\operatorname{perf}_Z(X,\mathcal A)$ the stricly full triangulated
    subcategory of $D_{\operatorname{Qcoh}}(\mathcal{A})$ consisting of those
    complexes supported on some closed subset $Z$ in $X$. If it is clear from context these will be written as $\operatorname{perf}(\mathcal{A})$ and $\operatorname{perf}_Z (\mathcal{A})$.
\end{definition}

An $\mathcal A$-module being `finitely generated projective', in the sense above, that it is locally a direct summand of a finite free $\mathcal A$-modules, is equivalent to it being finitely presented as $\mathcal A$-module and `locally projective' in the sense of \cite{Burban/Drozd/Gavran:2017}, i.e.\ its stalks are projective modules. 

Before we continue let us make the following observation.

\begin{lemma}\label{lem:structure_pushforward_reflects_0-objects}
    Let $(X,\mathcal A)$ be a  nc.\ scheme with structure map $\pi$. 
    The functor $\pi_\ast \colon D_{\operatorname{Qcoh}}(\mathcal{A}) \to D_{\operatorname{Qcoh}}(X)$ preserves and reflects ayclic complexes, in particular it reflects the zero object. 
\end{lemma}

\begin{proof}
    Exactness does not care about the module action.
\end{proof}

The following Proposition is a noncommutative version of \cite[Theorem 6.8]{Rouquier:2008}.

\begin{proposition}\label{prop:NCRouquier}
    Let $(X,\mathcal A)$ be a quasi-compact quasi-separated  nc.\ scheme.
    The perfect complexes on $(X,\mathcal A)$ are the compact objects of $D_{\operatorname{Qcoh}}(\mathcal{A})$, i.e.\ $D_{\operatorname{Qcoh}}(\mathcal{A})^c=\operatorname{perf}(\mathcal A)$. Moreover, if $Z$ is a closed subscheme of $X$ with quasi-compact complement, then $D_{\operatorname{Qcoh},Z} (\mathcal A)$ is generated by an object of $D_{\operatorname{Qcoh},Z} (\mathcal A) \cap \operatorname{perf} (\mathcal A)$.
\end{proposition}

\begin{proof}
    The first statement can be proved using Rouquier's cocovering, see e.g.\ \cite[Theorem 3.14]{Burban/Drozd/Gavran:2017} (they have the extra assumption that $X$ is separated, but it is not needed for the argument to go through in our set-up).

    To prove the second statement, let $G$ be an object of $D_{\operatorname{Qcoh},Z} (X) \cap \operatorname{perf}(X)$ that generates $D_{\operatorname{Qcoh},Z} (X)$, which exists by \cite[Theorem 6.8]{Rouquier:2008}.
    Let $\pi\colon (X,\mathcal A)\to X$ denote the structure morphism, we claim that $\mathbb L\pi^\ast G$ is the sought after object.
    Indeed, it is clearly contained in $D_{\operatorname{Qcoh},Z} (\mathcal A) \cap \operatorname{perf}(\mathcal A)$ and, moreover, suppose $E$ is an object of $D_{\operatorname{Qcoh},Z} (\mathcal{A})$ with, for all integers $n$, $\operatorname{Hom}_{\mathcal A}(\mathbb{L} \pi^\ast G,E[n])=0$. 
    By adjunction, it follows that $\operatorname{Hom}_X(G,\pi_\ast E[n])=0$ for all $n$.
    Hence $\pi_\ast E=0$ as $\operatorname{supp}(\pi_\ast E)$ is contained in $Z$, and so also $E=0$ by Lemma~\ref{lem:structure_pushforward_reflects_0-objects}. 
\end{proof}

\begin{remark}\label{rmk:pullback_cpt_gen_to_dqcoh}
    As the proof of Proposition~\ref{prop:NCRouquier} shows $\mathbb{L}\pi^\ast G$ is a compact generator for $D_{\operatorname{Qcoh}}(\mathcal{A})$ if $G$ is a compact generator for $D_{\operatorname{Qcoh}}(X)$.
\end{remark}

The following is a noncommutative version of Thomason-Trobaugh's localization sequence.

\begin{proposition}\label{prop:NCThomasonTrobaugh}
    Let $(X, \mathcal A)$ be a quasi-compact quasi-separated  nc.\ scheme. 
    If $j \colon U \to X$ is an open immersion with $U$ quasi-compact and $Z:=X\setminus U$, then there exists a Verdier localization sequence:
    \begin{displaymath}
        D_{\operatorname{Qcoh},Z}(\mathcal{A}) \to D_{\operatorname{Qcoh}}(\mathcal{A}) \xrightarrow{j^\ast} D_{\operatorname{Qcoh}}(j^\ast \mathcal{A}).
    \end{displaymath} 
\end{proposition}

\begin{proof}
    As $j^\ast$ has a fully faithful right adjoint\footnote{The adjoint preserving quasi-coherence requires the `topological' conditions in the statement.}, it follows that it is a Verdier localisation functor. It suffices to check that $\ker j^\ast = D_{\operatorname{Qcoh},Z}(\mathcal{A})$. But this is clear, as, for $E$ in $D_{\operatorname{Qcoh}}(\mathcal{A})$, we have
    \begin{displaymath}
        j^\ast E=0 \iff E_p = 0\text{ for all $p\in U$}
         \iff \operatorname{supp}(E)\subset X\setminus U=Z.\qedhere
    \end{displaymath}
\end{proof}

Combining this with Proposition \ref{prop:NCRouquier} and \cite[Theorem 2.1]{Neeman:1992} one obtains.

\begin{corollary}\label{cor:locsequenceperf}
    Let $(X, \mathcal A)$ be a quasi-compact quasi-separated  nc.\ scheme. 
    If $j \colon U \to X$ is an open immersion with $U$ quasi-compact and $Z:=X\setminus U$, then there exists a Verdier localization sequence up to summands:
    \begin{displaymath}
        \operatorname{perf}_Z (\mathcal{A}) \to \operatorname{perf} (\mathcal{A})\xrightarrow{j^\ast} \operatorname{perf} (j^\ast \mathcal{A}).
    \end{displaymath}
\end{corollary}

\begin{remark}\label{rmk:ELS_verdier_nc_scheme}
    Similarly, for $(X,\mathcal{A})$ a Noether nc.\ scheme and a closed subset $Z\subseteq X$ with associated open immersion $j \colon U:=X \setminus Z \to X$. There is a Verdier localization sequence:
    \begin{displaymath}
        D^b_{\operatorname{coh},Z}(\mathcal{A}) \to D^b_{\operatorname{coh}}(\mathcal{A}) \xrightarrow{j^\ast} D^b_{\operatorname{coh}}(j^\ast \mathcal{A}).
    \end{displaymath}
    The proof of this requires more work and can be found in \cite[Theorem 4.4]{Elagin/Lunts/Schnurer:2020}.
\end{remark}

As a brief intermezzo, let us give some instances of triangulated categories arising from nc.\ schemes admitting a classical or strong generator.

\begin{example}\label{ex:NCexamples}
    \hfill
    \begin{enumerate}
        \item Let $(X,\mathcal{A})$ be an Noether nc.\ scheme. If $X$ is $J\textrm{-}2$, then $D^b_{\operatorname{coh}}(\mathcal{A})$ admits a classical generator, cf. \cite[Theorem 4.15]{Elagin/Lunts/Schnurer:2020}. See \cite{Bhaduri/Dey/Lank:2023} for related work regarding bounds on Rouquier dimension of $D^b_{\operatorname{coh}}(X)$ for $X$ its underlying scheme.
        \item Let $X$ be a Noetherian scheme. Suppose $\mathcal{I}$ is a nilpotent ideal sheaf of $\mathcal{O}_X$ with $\mathcal{I}^n=(0)$ for some $n\geq 1$. Consider the coherent $\mathcal{O}_X$-algebra given by the following matrix:
        \begin{displaymath}
            \mathcal{A}:= \begin{pmatrix}
            \mathcal{O}_X & \mathcal{I} & \mathcal{I}^2 & \cdots & \mathcal{I}^{n-1} \\
            
            \mathcal{O}_X /\mathcal{I}^{n-1} & \mathcal{O}_X /\mathcal{I}^{n-1} & \mathcal{I}/\mathcal{I}^{n-1} & \cdots & \mathcal{I}^{n-2}/\mathcal{I}^{n-1} \\

            \mathcal{O}_X /\mathcal{I}^{n-2} & \mathcal{O}_X /\mathcal{I}^{n-2} & \mathcal{O}_X/\mathcal{I}^{n-2} & \cdots & \mathcal{I}^{n-3}/\mathcal{I}^{n-2} \\
            
            \vdots & \vdots & \vdots &  &\vdots \\
            
            \mathcal{O}_X /\mathcal{I} & \mathcal{O}_X /\mathcal{I} & \mathcal{O}_X /\mathcal{I} & \cdots & \mathcal{O}_X /\mathcal{I}
            \end{pmatrix}
        \end{displaymath}
        whose algebra structure is induced by its inclusion in $\operatorname{\mathcal{E}\! \mathit{nd}}(\oplus^n_{r=1}\mathcal{O}_X /\mathcal{I}^r)$. Algebras of this form are called \textbf{Auslander algebras}, cf.\ \cite[Remark 5.1]{Kuznetsov/Lunts:2015}. Such nc.\ schemes are equivalent to finite length \textbf{filtered schemes} equipped with $\mathcal{I}$-adic filtration, \cite[Proposition 6.16 and Remark 6.17]{DeDeyn:2023}. 
        There is a semiorthogonal decomposition of $D^b_{\operatorname{coh}}(\mathcal{A})$ whose components are triangle equivalent to $D^b_{\operatorname{coh}}(X_0)$ where $X_0$ is the closed subscheme defined by $\mathcal{I}$, cf.\ \cite[Corollary 5.15]{Kuznetsov/Lunts:2015} (in loc.\ cit.\ X is assumed to be separated and of finite type over a field, but this is not needed). It follows  that $D^b_{\operatorname{coh}}(\mathcal{A})$ admits a classical or strong generator if, and only if, $D^b_{\operatorname{coh}}(X_0)$ has such objects. 
        
        \item Let $A$ be a Noetherian ring that is finite over its center $Z(A)$ and assume that the center $Z(A)$ is of finite type over a perfect field $\mathsf{k}$. When considered as a  differential graded category $D^b(\operatorname{mod}(A))$ is smooth over $\mathsf{k}$ by \cite[Theorem 5.1]{Elagin/Lunts/Schnurer:2020}. Consequently, $D^b(\operatorname{mod}(A))$ admits a strong generator, see e.g.\  \cite[Lemmas 3.5 and 3.6]{Lunts:2010}. 
        See \cite[Remark 2.6]{Elagin/Lunts/Schnurer:2020} for details.
    \end{enumerate}
\end{example}

To end this section, we make a small remark concerning the countable Rouquier dimension for nc.\ schemes. 
Recall that the \textbf{countable Rouquier dimension} of a triangulated category $\mathcal{T}$, denoted by $\operatorname{cdim}\mathcal{T}$ and introduced in \cite{Olander:2023}, is the smallest integer $d$ such that one has $\langle \mathcal{C} \rangle_{d+1} = \mathcal{T}$ for $\mathcal{C}$ a subcategory of $\mathcal{T}$ with countably many objects. 

\begin{remark}
    When $X$ is a regular Noetherian scheme of finite Krull dimension admitting an ample line bundle $\mathcal{L}$, then $\operatorname{cdim}D^b_{\operatorname{coh}}(X)$ is at most $\dim X$ by \cite[Theorem 4]{Olander:2023}. 
    In fact, the proof of loc.\ cit., when coupled with \cite[Corollary 6.5]{Burban/Drozd/Gavran:2017}, goes through for any Noether nc.\ scheme $(X,\mathcal{A})$ when $X$ admits an ample line bundle and the global dimension of $\mathcal{A}_p$ is uniformly bounded for all points $p$ in $X$.
    It is even possible to remove the ample line bundle assumptions, instead working with an ample family of line bundles (which e.g.\ exist on regular Noetherian schemes with affine diagonal).
\end{remark}

\section{Key ingredients}
\label{sec:key_ingredients}

This section establishes some technical results that are used to prove the main theorem. We start with the following, which yields new examples of approximable triangulated categories, and should be compared to \cite[Example 3.6]{Neeman:2021b}.

\begin{proposition}\label{prop:approx_for_dqcoh_coherent_algebra}
    Let $(X,\mathcal{A})$ be a quasi-compact separated nc.\ scheme.
    The triangulated category $D_{\operatorname{Qcoh}}(\mathcal{A})$ is approximable.
\end{proposition}

\begin{proof}
    Let $\pi \colon (X,\mathcal A)\to X$ denote the structure morphism.
    The triangulated category $D_{\operatorname{Qcoh}}(X)$ is approximable, cf. \cite[Example 3.6]{Neeman:2021b}.
    Thus let $G_X$ be a compact generator for $D_{\operatorname{Qcoh}}(X)$ satisfying the definition of approximability, and let $G:=\mathbb{L}\pi^\ast G_X$ be the corresponding compact generator in $D_{\operatorname{Qcoh}}(\mathcal{A})$, cf.\ Remark \ref{rmk:pullback_cpt_gen_to_dqcoh}.
    As $\mathbb{L}\pi^\ast D_{\operatorname{Qcoh}}^{\leq 0}(X)$ is contained in $D_{\operatorname{Qcoh}}^{\leq 0}(\mathcal A)$ and $\pi_\ast D_{\operatorname{Qcoh}}^{\leq 0}(\mathcal A)$ is contained in $D_{\operatorname{Qcoh}}^{\leq 0}(X)$, $G$ satisfies the first two conditions in the definition of approximability (although, of course, this is true for any perfect complex), see Definition \ref{def:approx}.
    
    For the third, let $F$ be an object of $D_{\operatorname{Qcoh}}^{\leq 0}(\mathcal{A})$. Approximability of $D_{\operatorname{Qcoh}}(X)$ gives a distinguished triangle
    \begin{displaymath}
        E \to \pi_\ast F \to D \to E[1]\quad\text{(in $D_{\operatorname{Qcoh}}(X)$)},
    \end{displaymath}
    where $E$ belongs to $\overline{\langle G_X \rangle}^{[-A,A]}_A$ and $D$ is in $ D_{\operatorname{Qcoh}}^{\leq -1}(X)$.
    Applying $\mathbb{L}\pi^\ast$ yields the distinguished triangle
    \begin{displaymath}
        \mathbb{L}\pi^\ast E \to \mathbb{L}\pi^\ast\pi_\ast F \to \mathbb{L}\pi^\ast D \to \mathbb{L}\pi^\ast E[1]\quad\text{(in $D_{\operatorname{Qcoh}}(\mathcal{A})$)},
    \end{displaymath}
    where $\mathbb{L}\pi^\ast E$ is in $\overline{\langle G \rangle}^{[-A,A]}_A$ and $\mathbb{L}\pi^\ast D$ is in $D_{\operatorname{Qcoh}}^{\leq -1}(\mathcal{A})$. Consider next the counit of adjunction $\varepsilon_F \colon \mathbb{L}\pi^\ast\pi_\ast F \to F$. As $F$ is an object of $D_{\operatorname{Qcoh}}^{\leq 0}(\mathcal{A})$, it can be represented by a complex with all positive terms vanishing, and we can do so as well with $\pi_\ast F$. Choose a flat resolution $F' \to \pi_\ast F$, which is surjective at the zeroth degree as a map of complexes. There exists a quasi-isomorphism $\mathbb{L}\pi^\ast \pi_\ast F \cong \pi^\ast F'$ in $D_{\operatorname{Qcoh}}(\mathcal{A})$ and $\varepsilon_F$ factors as $\pi^\ast F' \to \pi^\ast \pi_\ast F \to F$. Both the maps are surjective at the zeroth degree --- the first as tensor product is right exact, and the second as the $\pi_\ast$ is faithful at the level of abelian categories. It follows that $\varepsilon_F $ is surjective at the zeroth degree as a map of complexes, and hence, on the zeroth cohomology as both of these complexes are concentrated in non-positive degrees. Thus $\operatorname{cone}(\varepsilon_F)$ belongs to $D_{\operatorname{Qcoh}}^{\leq -1}(\mathcal{A})$. Therefore, if we apply the octahedral axiom to $\mathbb{L}\pi^\ast E \to \mathbb{L}\pi^\ast\pi_\ast F \xrightarrow{\varepsilon_F} F$, we get the required approximating distinguished triangle, which completes the proof.
\end{proof}

The next proposition shows that for $\mathcal{T}:=D_{\operatorname{Qcoh}}(\mathcal{A})$ with $(X,\mathcal{A})$ a Noether nc.\ scheme one has $\mathcal{T}^-_c=D^-_{\operatorname{coh}}(\mathcal{A})$ (notation from Definition~\ref{def:T^-_c}).
Here, of course, $D^{-}_{\operatorname{coh}}(\mathcal{A})$ denotes those objects $E$ in $D_{\operatorname{Qcoh}}(\mathcal{A})$ such that $\mathcal{H}^j (E)=0$ for $j \gg 0$ and $\mathcal{H}^n (E)$ is coherent for all integers $n$. 
It is a noncommutative version of what is known as `approximation by perfect complexes' on a scheme, cf.\ \cite[Theorem 4.1]{Lipman/Neeman:2007}.

\begin{proposition}\label{prop:coherenet_alg_approx_perfect}
    Let $(X,\mathcal{A})$ be a Noether nc.\ scheme. 
    For any integer $n$ and any object F in $D^{-}_{\operatorname{coh}}(X)$, there exists a triangle $E \to F \to D \to E[1]$ with $E$ in $\operatorname{perf}(\mathcal{A})$ and $D$ in $D_{\operatorname{Qcoh}}(\mathcal{A})^{\leq -n}$.
\end{proposition}

\begin{proof}
    Let $F$ be any object in $D^{-}_{\operatorname{coh}}(\mathcal{A})$. Then, by shifting if necessary, we can assume that $F$ lies in $D_{\operatorname{Qcoh}}(\mathcal{A})^{\leq 0} \cap D^{-}_{\operatorname{coh}}(X)$. We will now construct the required triangles inductively on $n \geq 1$. 
    
    The case $n=1$ is proved by exactly the same argument as in Proposition \ref{prop:approx_for_dqcoh_coherent_algebra}. 
    Thus it just remains to prove the inductive step. So, suppose we know the result up to some positive integer $n$. That is, there exists a triangle $E \to F \to D_n \to E[1]$ with $E$ in $\operatorname{perf}(\mathcal{A})$ and $D_n$ in $D_{\operatorname{Qcoh}}(\mathcal{A})^{\leq -n}$. As $E$ lies in $\operatorname{perf}(\mathcal{A}) \subseteq D^{-}_{\operatorname{coh}}(\mathcal{A})$, we have that $D_n$ lies in $D_{\operatorname{Qcoh}}(\mathcal{A})^{\leq -n} \cap D^{-}_{\operatorname{coh}}(\mathcal{A})$. So we can apply the $n=1$ case to $D_n[-n]$ and obtain a triangle $E' \to D_n \to D_{n+1} \to E'[1]$ with $E'$ in $\operatorname{perf}(\mathcal{A})$ and $D_{n+1}$ in $D_{\operatorname{Qcoh}}(\mathcal{A})^{\leq -n-1}$. Applying the octahedral axiom to the composable morphisms $F \to D_n \to D_{n+1}$ yields the result.  
\end{proof}

\begin{remark}
    One can modify the above proof to work for any quasi-compact quasi-separated scheme $X$ as long as $\mathcal{A}$ is pseudo-coherent as a $\mathcal{O}_X$-module, granted one replaces $D^-_{\operatorname{coh}}(\mathcal{A})$ by the pseudo-coherent $\mathcal{A}$-complexes.
\end{remark}

\begin{corollary}\label{cor:rouq_dim_via_approx_perf_alg}
    Let $(X,\mathcal{A})$ be an Noether nc.\ scheme. 
    Suppose $G$ is an object of $D^b_{\operatorname{coh}}(\mathcal{A})$ and $n\geq 0$.
    Then
    \begin{enumerate}
        \item $\operatorname{perf}(\mathcal{A})\subseteq\langle G \rangle_n$ if, and only if, $D^b_{\operatorname{coh}}(\mathcal{A}) = \langle G  \rangle_n$,
        \item $\overline{\langle G \rangle}_n \cap D^b_{\operatorname{coh}}(\mathcal{A}) = \langle G \rangle_n$.
    \end{enumerate}
\end{corollary}

\begin{proof}
    Both claims follow from \cite[Proposition 3.5]{Lank/Olander:2024}, noting that the `approximation property' is satisfied for $\operatorname{perf}(\mathcal{A})$ by Proposition~\ref{prop:coherenet_alg_approx_perfect}, and by viewing  $D_{\operatorname{Qcoh}}(\mathcal{A})$ with its standard $t$-structure. (Also using Lemma~\ref{lem:big_coprod_intersection_big_thick} for the second claim.)
\end{proof}

\begin{lemma}\label{lem:mv_distinguished triangle_quasi_coh_algebra}
    Let $(X,\mathcal{A})$ be a quasi-compact quasi-separated nc.\ scheme. Suppose $U$ and $V$ are quasi-compact open subsets of $X$ with associated open immersions $i \colon U \to X$, $j \colon V \to X$, and $k \colon U \cap V \to X$. 
    For any object $F$ of $D_{\operatorname{Qcoh}}(\mathcal{A})$ there exists a Mayer-Vietoris distinguished triangle in $D_{\operatorname{Qcoh}}(\mathcal{A})$: 
    \begin{displaymath}
        F \to \mathbb{R}j_\ast j^\ast F\oplus \mathbb{R}i_\ast i^\ast F \to \mathbb{R}k_\ast k^\ast F \to F[1].
    \end{displaymath}
\end{lemma}

\begin{proof}
    The proof of \cite[\href{https://stacks.math.columbia.edu/tag/08BV}{Tag 08BV}]{StacksProject} works in this generality, yielding a distinguished triangle in $D(\mathcal A)$.
    By the quasi-compact and quasi-separated hypotheses this restricts to $D_{\operatorname{Qcoh}}(\mathcal A)$.
\end{proof}

\begin{lemma}\label{lem:Cech_resolution}
    Let $(X,\mathcal{A})$ be a quasi-compact quasi-separated nc.\ scheme. Suppose $V$ is a quasi-compact open subset of $X$ with associated open immersion $j \colon V \to X$.
    If $\pi \colon (X,\mathcal{A})\to X$ denotes the structure morphism, then $\mathbb{R}j_\ast j^\ast \mathcal{A}$ is quasi-isomorphic to $\mathbb{L}\pi^\ast\mathbb{R}j_\ast j^\ast \mathcal{O}_X$ in $D_{\operatorname{Qcoh}}(\mathcal{A})$.
\end{lemma}

The following two lemmas are special instances of what may be considered `flat base change' for noncommutative schemes.

\begin{proof}
    The structure morphism $\mathcal{O}_X\to \pi_\ast\mathcal{A}$ gives us a morphism $\mathbb{L}\pi^\ast\mathbb{R}j_\ast j^\ast \mathcal{O}_X \to \mathbb{R}j_\ast j^\ast \mathcal{A}$. 
    Checking that this is a quasi-isomorphism is local on $X$, so we may assume $X=\operatorname{Spec} R$ is affine and the complement of $V$ is generated by $f_1,\dots, f_l$ in $R$.
    The global sections of $\mathcal{A}$ give us a $R$-algebra, which we denote by $A$.
    
    Let $C$ denote the following complex of $R$-modules
    \begin{displaymath}
        \prod_{i}R_{f_i} \to \prod_{i_1 < i_2}R_{f_{i_1}f_{i_2}} \to \dots \to R_{f_1\dots f_l}
    \end{displaymath}
    and note that $C$ is quasi-isomorphic to $\mathbb{R}j_\ast j^\ast \mathcal{O}_X$ in $D^b_{\operatorname{coh}}(X)$ as it is just the \v{C}ech resolution of $\mathbb{R}j_\ast(\mathcal{O}_V)$. This is a bounded complex of flat $R$-modules, so $\mathbb{L}\pi^\ast\mathbb{R}j_\ast j^\ast \mathcal{O}_X=\pi^\ast C$ is the following complex of $A$-modules:
    \begin{displaymath}
        \prod_{i} A_{f_i} \to \prod_{i_1 < i_2} A_{f_{i_1}f_{i_2}} \to \dots \to A_{f_1\dots f_l}.
    \end{displaymath}
    But this is exactly the \v{C}ech resolution of $\mathbb{R}j_\ast(j^\ast\mathcal{A})$, and hence is quasi-isomorphic to $\mathbb{R}j_\ast j^\ast \mathcal{A}$, which is what we needed to show.
\end{proof}

\begin{lemma}\label{lem:flat_base_change}
    Let $(X,\mathcal{A})$ be a quasi-compact quasi-separated nc.\ scheme admitting an open cover $X=U \cup V$ and denote the associated open immersions by 
    \begin{equation}\label{eq:openimm}
        \begin{tikzcd}
            {U\cap V} & U \\
        	V & X\rlap{ .}
        	\arrow["{j^{\prime}}", from=1-1, to=1-2]
        	\arrow["{h^{\prime}}"', from=1-1, to=2-1]
        	\arrow["h", from=1-2, to=2-2]
        	\arrow["j"', from=2-1, to=2-2]
            \arrow["g"{description}, from=1-1, to=2-2]
        \end{tikzcd}
    \end{equation}
    Then $\mathbb{R}j^\prime_\ast j^{\prime,\ast} h^\ast \mathcal{A}$ is quasi-isomorphic to $h^\ast \mathbb{R}j_\ast j^\ast \mathcal{A}$ in $D_{\operatorname{Qcoh}}(h^* \mathcal{A})$.
\end{lemma}

\begin{proof}
    Consider an injective resolution $I$ of $j^\ast \mathcal{A}=\mathcal{A}|_{V}$ (in $D_{\operatorname{Qcoh}}(\mathcal{A}|_{ V})$) . 
    As $h^{\prime, \ast} I$ is an injective resolution of $h^{\prime, \ast}j^{\ast} \mathcal{A}= j^{\prime, \ast}h^{\ast} \mathcal{A}=\mathcal{A}|_{U\cap V}$ (in $D_{\operatorname{Qcoh}}(\mathcal{A}|_{U\cap V})$) we have the string of quasi-isomorphisms
    \begin{displaymath}
        \mathbb{R}j^\prime_\ast j^{\prime,\ast} h^\ast \mathcal{A} \cong \mathbb{R}j^\prime_\ast h^{\prime,\ast} j^\ast \mathcal{A} \cong j^\prime_\ast h^{\prime,\ast} I \cong h^\ast j_\ast I \cong h^\ast \mathbb{R}j_\ast j^\ast \mathcal{A}\quad\text{(in $D_{\operatorname{Qcoh}}(h^\ast \mathcal{A})$),}
    \end{displaymath}
    where we used a very special case of flat base change, which can also simply be checked as it involves no derived functors.
\end{proof}

The following is a noncommutative variant of \cite[Theorem 6.2]{Neeman:2021} where the result was shown for the special case $\mathcal{A}=\mathcal{O}_X$.

\begin{proposition}\label{prop:compact_boundedness_noncommuative}
    Let $(X,\mathcal{A})$ be a quasi-compact separated nc.\ scheme and $j \colon V \to X$ be an open immersion with $V$ a quasi-compact open subset of $X$. Suppose that $G$ is a compact generator for $D_{\operatorname{Qcoh}} (\mathcal{A})$. If $P$ is in $\operatorname{perf}(j^\ast \mathcal{A})$, then there exists an integer $N\geq 0$ such that $\mathbb{R}j_\ast P$ is in $\overline{\langle G \rangle}_N$.
\end{proposition}

\begin{proof}
    As usual let $\pi \colon (X,\mathcal A)\to X$ denote the structure morphism (and similarly for its restriction to open subsets of $X$).
    
    As $X$ is quasi-compact, we can find an affine open cover $X = \cup_1^n U_i$. Set $X_l = V \cup (\cup_1^l U_i)$ where $1\leq l \leq n$. The map $j \colon V \to X$ factors as
    \begin{displaymath}
        V =: X_0 \to X_1 \to \dots \to X_n = X,
    \end{displaymath}
    and so, it is enough to prove the desired claim for $X = U \cup V$ where $U=\operatorname{Spec}(R)$ is an affine open subscheme. 
    Label the various open immersions as in \eqref{eq:openimm}. Note that $U\cap V$ is a quasi-compact open subscheme of $U$. So there exist finitely many elements $f_1,\dots,f_l \in R$ such that
    \begin{displaymath}
    U \cap V = \bigcup_{i=1}^{l} \operatorname{Spec}(R_{f_i}) \subseteq \operatorname{Spec}(R) = U.
    \end{displaymath}
    First, we prove the statement for $P= j^\ast \mathcal{A}$. There is a distinguished triangle in $D_{\operatorname{Qcoh}}(\mathcal{A})$ obtained from the unit of adjunction:
    \begin{equation}\label{eq:cbnc1}
        Q \to  \mathcal{A}  \to \mathbb{R}j_\ast j^\ast \mathcal{A} \to Q[1].
    \end{equation}
    Similarly there is a distinguished triangle in $D_{\operatorname{Qcoh}}(U) \cong D(\operatorname{Mod}(R))$ also coming from the corresponding unit of adjunction:
    \begin{equation}\label{eq:cbnc1'}
        Q^\prime \to \mathcal{O}_U \to \mathbb{R}j^\prime_\ast j^{\prime,\ast} \mathcal{O}_{U} \to Q^\prime [1].
    \end{equation}
    We show that $Q^\prime$ has a finite projective resolution in $D(\operatorname{Mod}(R))$. The map $\mathcal{O}_U \to \mathbb{R}j^\prime_\ast j^{\prime,\ast} \mathcal{O}_{U}$ is exactly the map $R \to C$ in $D(\operatorname{Mod}(R))$, where $C := \Pi_{i}R_{f_i} \to \Pi_{i_1 < i_2}R_{f_{i_1}f_{i_2}} \to \dots \to R_{f_1\dots f_l}$ is just the \v{C}ech resolution of $\mathbb{R}j^\prime_\ast j^{\prime,\ast} \mathcal{O}_{U}$. So, $Q^\prime$ is exactly the complex given by $R \to \Pi_{i}R_{f_i} \to \Pi_{i_1 < i_2}R_{f_{i_1}f_{i_2}} \to \dots \to R_{f_1\dots f_l}$. This immediately gives us that $Q^\prime \cong \otimes_{i=1}^{l}Q^\prime_i$, where $Q^\prime_i := R \to R_{f_i}$ for each $i \leq i \leq l$. Now, it is straightforward to verify that $Q^\prime_i \cong \varinjlim_n (R \xrightarrow{f_i^n} R)$ in the category of chain complexes. And so, $Q^\prime_i$ is the homotopy colimit of the complexes $R \xrightarrow{f_i^n} R$, and thus has a finite projective resolution. Hence,  $Q^\prime \cong \otimes_{i=1}^{l}Q^\prime_i$ has a finite projective resolution, that is, it is quasi-isomorphic to a bounded complex of projective modules in $D(\operatorname{Mod}(R))$.
    
    Pulling back the above triangle \ref{eq:cbnc1'} via $\pi^\ast$ yields a distinguished triangle in $D_{\operatorname{Qcoh}}(h^\ast \mathcal{A})$:
    \begin{equation}\label{eq:cbnc2}
        \mathbb{L}\pi^\ast Q^\prime \to h^\ast\mathcal{A}  \to \mathbb{L}\pi^\ast \mathbb{R}j^\prime_\ast j^{\prime,\ast} \mathcal{O}_{U} \to \pi^\ast Q^\prime [1].
    \end{equation}
    By Lemmas \ref{lem:Cech_resolution} and \ref{lem:flat_base_change} we have quasi-isomorphisms
    \begin{displaymath}
        \mathbb{L}\pi^\ast \mathbb{R}j^\prime_\ast j^{\prime,\ast} \mathcal{O}_{U}\cong \mathbb{R}j^\prime_\ast {j^\prime}^\ast h^* \mathcal{A}\cong h^\ast \mathbb{R}j_\ast j^\ast \mathcal{A}
    \end{displaymath}
    in $D_{\operatorname{Qcoh}}(h^\ast \mathcal{A})$.
    As both $h^\ast Q$ and $\mathbb{L}\pi^\ast Q'$ are supported on $X\setminus Z$, the distinguished triangle \eqref{eq:cbnc2} is isomorphic to
    \begin{displaymath}
        h^\ast Q \to  h^\ast \mathcal{A}  \to h^\ast \mathbb{R}j_\ast j^\ast \mathcal{A} \to h^\ast Q[1]\quad\text{(in $D_{\operatorname{Qcoh}}(h^\ast \mathcal{A})$}).
    \end{displaymath}
    Consequently, $h^\ast Q$ is a quasi-isomorphic to $\mathbb{L}\pi^\ast Q^\prime$ (in $D_{\operatorname{Qcoh}}(h^\ast \mathcal{A})$). 
    
    This implies that $h^\ast Q \cong \pi^\ast Q^\prime$ can be equally be represented by a bounded complex of projective modules in $D_{\operatorname{Qcoh}}(h^\ast \mathcal{A})$, and so, there exists an integer $q > 0$ such that
    \begin{equation}\label{eq:cbnc3}
        \operatorname{Hom}\big(h^\ast Q, D_{\operatorname{Qcoh}}^{\leq -q}(h^*\mathcal{A})\big) = 0.
    \end{equation}
    
    Let $F$ be an object of $D_{\operatorname{Qcoh}}(\mathcal{A})$. By Lemma \ref{lem:mv_distinguished triangle_quasi_coh_algebra} we have a Mayer-Vietoris distinguished triangle
    \begin{displaymath}
        F \to \mathbb{R}j_\ast j^\ast F \oplus \mathbb{R}h_\ast h^\ast F \to \mathbb{R}g_\ast g^\ast F \to F[1].
    \end{displaymath}
    Applying the functor $\operatorname{Hom}(Q,-)$ to this distinguished triangle, yields a long exact sequence. 
    Since the following terms vanish, for any integer $n$,
    \begin{displaymath}
        \begin{aligned}
            \operatorname{Hom}(Q, \mathbb{R}j_\ast j^\ast F[n])
            &=\operatorname{Hom}(j^\ast Q,  j^\ast F[n]) 
            \\&= \operatorname{Hom}( g^\ast Q, g^\ast F) 
            \\&= \operatorname{Hom}(Q,\mathbb{R}g_\ast g^\ast F[n])=0,
        \end{aligned}
    \end{displaymath}    
    as $Q$ is supported on $X \setminus V$, we obtain
    \begin{displaymath}
        \operatorname{Hom}(Q,F) \cong \operatorname{Hom}(Q,\mathbb{R}h_\ast h^\ast F) \cong \operatorname{Hom}( h^\ast Q,h^\ast F).
    \end{displaymath}
    Hence this $\operatorname{Hom}$ vanishes if $F$ is an object of $D_{\operatorname{Qcoh}}^{\leq -q}(\mathcal{A})$ by Equation \eqref{eq:cbnc3}.
    
    As $\mathbb{R}j_\ast j^\ast \mathcal{A}$ belongs to $D_{\operatorname{Qcoh}}^{\leq t}(\mathcal{A})$ for some $t \geq 0$ by Lemma~\ref{lem:pushforward_cohomologically_bounded}, $Q$ belongs to $D_{\operatorname{Qcoh}}^{\leq t+1}(\mathcal{A})$.
    Using approximability of $D_{\operatorname{Qcoh}}(\mathcal{A})$, see Proposition~\ref{prop:approx_for_dqcoh_coherent_algebra}, we obtain a distinguished triangle $E \to Q \to D \to E[1]$ with $E$ is in $\overline{\langle G \rangle}_{N^\prime}$ and $D$ is in $D_{\operatorname{Qcoh}}(\mathcal{A})^{\leq -q}$ for some $N'>0$.
    By the previous paragraph map $Q \to D$ vanishes, and so, $Q$ belongs to $\overline{\langle G \rangle}_{N^\prime}$. 
    As $\mathcal{A}$ is a compact object, there exists $N^{\prime \prime} > 0$ such that $\mathcal{A}$ is in $\overline{\langle G \rangle}_{N^{\prime \prime}}$. 
    Hence, we obtain that $\mathbb{R}j_\ast j^\ast \mathcal{A}$ is in $\overline{\langle G \rangle}_N$ for some $N\geq 0$ by \eqref{eq:cbnc1}.

    Lastly, we prove the claim for an arbitrary object $P$ in $\operatorname{perf}(j^\ast \mathcal{A})$. Assume $j^\prime,h,g$ are as previously defined. Let $P$ be in $\operatorname{perf}(j^\ast \mathcal{A})$. By Corollary~\ref{cor:locsequenceperf}, there is an object $P^\prime$ in $\operatorname{perf}(\mathcal{A})$ such that $j^\ast P^\prime$ finitely builds $P$ in one cone in $D_{\operatorname{Qcoh}}(j^\ast \mathcal{A})$. Hence, we see that $\mathbb{R}j_\ast P$ belongs to $\overline{\langle \mathbb{R}j_\ast j^\ast P^\prime \rangle}_1$, and so it suffices to show that there exists $L \geq 0$ such that $\mathbb{R}j_\ast j^\ast P^\prime$ belongs to $\overline{\langle G \rangle}_L$. The work above tells us there exists $N_1 \geq 0$ such that $\mathbb{R}g_\ast g^\ast \mathcal{A}$ is in $\overline{\langle G \rangle}_{N_1}$. Note that $g^\ast P^\prime$ is in $\operatorname{perf}(g^\ast \mathcal{A})$, and $g^\ast \mathcal{A}$ is a compact generator for $D_{\operatorname{Qcoh}}(g^\ast \mathcal{A})$ as $U\cap V$ is quasi-affine, so we can choose an $N_2\geq 0$ such that $g^\ast P^\prime$ is in $\overline{\langle g^\ast \mathcal{A} \rangle}_{N_2}$. This tells us that $\mathbb{R}g_\ast g^\ast P^\prime$ is in $\overline{\langle \mathbb{R} g_\ast g^\ast \mathcal{A} \rangle}_{N_2}$. It follows that $\mathbb{R}g_\ast g^\ast P^\prime$ is in $\overline{\langle G \rangle}_{N_1 N_2}$. Choose $N_3$ such that $P^\prime$ belongs to $\overline{\langle G \rangle}_{N_3}$, which exists as $P^\prime$ is in $\langle G \rangle$. 
    Again using a Mayer-Vietoris distinguished triangle in $D_{\operatorname{Qcoh}}(\mathcal{A})$
    \begin{displaymath}
        P^\prime \to \mathbb{R}j_\ast j^\ast P^\prime \oplus \mathbb{R}h_\ast h^\ast P^\prime \to \mathbb{R}g_\ast g^\ast P^\prime \to P^\prime [1].
    \end{displaymath}    
    and noting that both $P^\prime$ and $\mathbb{R}g_\ast g^\ast P^\prime$ belong to $\overline{\langle G \rangle}_N$ where $N$ is largest of the $N_1 N_2, N_3$, ensures $\mathbb{R}j_\ast j^\ast P^\prime \oplus \mathbb{R}h_\ast h^\ast P^\prime$ and its direct summands are in $\overline{\langle G \rangle}_{2N}$. This completes the proof.
\end{proof}

\section{Applications}
\label{sec:applications}

\subsection{Regularity}
\label{sec:applications_regularity}

The following can be viewed as a noncommutative analog of \cite[Theorem 0.5]{Neeman:2021}, which asserted a relationship to between a quasi-compact separated scheme $X$ being covered by affine spectra of finite global dimension to $\operatorname{perf}(X)$ admitting a strong generator.

\begin{theorem}\label{thm:coherent_alg_perf_strong_gen}
    Let $(X,\mathcal{A})$ be a quasi-compact separated nc.\ scheme. The following are equivalent:
    \begin{enumerate}
        \item There exists a perfect strong $\oplus$-generator for $D_{\operatorname{Qcoh}}(\mathcal{A})$.
        \item For each affine open $U \subseteq X$, $D_{\operatorname{Qcoh}}(\mathcal{A}|_U)$ has a perfect strong $\oplus$-generator.
        \item $ \mathcal{A}(U)$ has finite global dimension for each affine open $U \subseteq X$.
        \item $ \mathcal{A}(U_i)$ has finite global dimension for each open $U_i$ in an affine open cover $X=\cup^n_{i=1}U_i$. 
    \end{enumerate}
    Any of these equivalent conditions imply that $\operatorname{perf}(\mathcal{A})$ has finite Rouquier dimension.
\end{theorem}

\begin{proof}
    It is evident that $(1)\implies (2)$ using Proposition~\ref{prop:NCThomasonTrobaugh}. Moreover, $(2) \implies (3)$ follows by \cite[Proposition 7.25 $(iii) \implies (i)$]{Rouquier:2008}, and $(3)\implies (4)$ holds trivially. So, it remains to prove $(4)\implies (1)$. 
    For this, suppose $\mathcal{A}(U_i)$ has finite global dimension for each open immersion $s_i \colon U_i \to X$ in an affine open cover $X=\cup^n_{i=1}U_i$. An inductive argument on the number of components of an affine open cover coupled with the Mayer-Vietoris distinguished triangles of Lemma~\ref{lem:mv_distinguished triangle_quasi_coh_algebra} tells us that there exists an integer $N\geq 0$ such that $D_{\operatorname{Qcoh}}(\mathcal{A})=\overline{\langle \oplus^n_{i=1} \mathbb{R} s_{i,\ast} s^\ast_i \mathcal{A} \rangle}_N$. However, if $G$ is a compact generator for $D_{\operatorname{Qcoh}}(\mathcal{A})$, then Proposition~\ref{prop:compact_boundedness_noncommuative} ensures there exists an $L\gg 0$ such that $\oplus^n_{i=1} \mathbb{R} s_{i,\ast} s^\ast_i \mathcal{A}$ belongs to $\overline{\langle G \rangle}_L$. Hence, we see that $\overline{\langle G \rangle}_T = D_{\operatorname{Qcoh}}(\mathcal{A})$ for some $T\geq 0$, which is what we needed to show.

    To see the last claim, note that if there exists a perfect complex $G$ and an integer $T \geq 1$ with $D_{\operatorname{Qcoh}}(\mathcal{A}) = \overline{\langle G \rangle}_T$, then $\operatorname{perf}(\mathcal{A}) = \langle G \rangle_T$ by \cite[Proposition 2.2.4]{Bondal/VandenBergh:2003}.
    By definition this means that $G$ is a strong generator for $\operatorname{perf}(\mathcal{A})$.
\end{proof}

\begin{corollary}\label{cor:alg_perf_strong_gen_weak_version}
     Let $(X,\mathcal{A})$ be a quasi-compact separated nc.\ scheme with $\mathcal{A}(U)$ a Noetherian ring for each affine open $U\subseteq X$. The following are equivalent:
     \begin{enumerate}
        \item There exists a strong generator for $\operatorname{perf}(\mathcal{A})$.
        \item For each affine open $U \subseteq X$, $\operatorname{perf}(\mathcal{A}|_U)$ has a strong generator.
        \item $\mathcal{A}(U)$ has finite global dimension for each affine open $U \subseteq X$.
        \item $ \mathcal{A}(U_i)$ has finite global dimension for each open $U_i$ in an affine open cover $X=\cup^n_{i=1}U_i$. 
    \end{enumerate}
\end{corollary}

\begin{proof}
    Since for a Noetherian ring having finite global dimension is equivalent to the perfect complexes having a strong generator, see e.g.\ from \cite[Proposition 10]{Krause:2024} or \cite[Corollary 4.3.13]{Letz:2020}, the result follows from Theorem \ref{thm:coherent_alg_perf_strong_gen} (using Corollary \ref{cor:locsequenceperf}).
\end{proof}

\begin{remark}\label{rmk:stong_gen_to_strong_oplus_for_perf}
    With the set-up as in Theorem \ref{thm:coherent_alg_perf_strong_gen} the existence of a perfect strong $\oplus$-generator as in the statement can be checked affine locally.
    In addition, in the set-up of Corollary \ref{cor:alg_perf_strong_gen_weak_version} the existence of a strong generator for the category of perfect complexes can also be checked affine locally.
    Indeed, the respective statements show that this is an affine local condition.
\end{remark}

\subsection{Azumaya algebras}
\label{sec:applications_azumaya}

As a first application we prove a result for strong generation of Azumaya nc.\ schemes.
One important consequence of Corollary~\ref{cor:alg_perf_strong_gen_weak_version}, that we exploit, is that it reduces arguments to the affine case and `easier' homological algebra.
More applications can be found in \cite{DeDeyn/Lank/ManaliRahul:2024c}

\begin{corollary}\label{cor:azumaya_detects_regular}
    If $X$ is a separated Noetherian scheme, then the following are equivalent:
    \begin{enumerate}
        \item $\operatorname{perf}(X)$ has finite Rouquier dimension, i.e.\ $X$ is regular of finite Krull dimension.
        \item $\operatorname{perf}(\mathcal{A})$ has finite Rouquier dimension for all Azumaya algebras $\mathcal{A}$ over $X$.
        \item $\operatorname{perf}(\mathcal{A})$ has finite Rouquier dimension for some Azumaya algebra $\mathcal{A}$ over $X$.
    \end{enumerate}
\end{corollary}

\begin{proof}
    Let $\mathcal{A}$ be any Azumaya algebra over $X$ and denote the structure morphism by $\pi$ as usual.
    By Corollary~\ref{cor:alg_perf_strong_gen_weak_version} the statements are local on $X$, so we can reduce to $X=\operatorname{Spec} R$ being affine and $A:=\mathcal{A}(X)$ finite free as $R$-module. 
    Moreover, by loc.\ cit.\ it suffices to prove $A$ has finite global dimension if, and only if, $R$ has finite global dimension. First, suppose $A$ has finite global dimension $n$.
    Let $M$ be any $R$-module. 
    By assumption $\pi^\ast M$ has projective dimension at most $n$. As, again by assumption, the restriction of a projective $A$-module is projective as $R$-module and as restriction is exact, it follows that $\pi_\ast \pi^\ast M$ has projective dimension at most $n$.
    Therefore, since $\pi_\ast \pi^\ast M$ is a finite direct sum of copies of $M$, as $A$ is $R$-free, the same is true for $M$. Conversely, assume $R$ has finite global dimension.
    It then follows immediately from \cite[Theorem 1.8]{Auslander/Goldman:1961} that the global dimension of $A$ is also finite.
\end{proof}

\begin{remark}\label{rmk:global_dim_azumaya_algebra}
    In fact the proof shows that the global dimension of an Azumaya algebra is the same as that of its center. 
    This is really the underlying reason for this result. 
\end{remark}

\begin{remark}
    A fun alternative argument for the converse direction is the following.
    With notation as in the proof, note that $A^e:=A^{op}\otimes_R A$ is Morita equivalent to $R$, see e.g.\ \cite[Theorem 3.4]{DeMeyerIngraham:1971}.
    Hence if $R$ has finite global dimension, so does $A^e$.
    Consequently $A$ has a finite projective resolution $P^\bullet$ of $A^e$-modules.
    The claim then follows as $(-\otimes_A P^\bullet)=(-\otimes_A A)$ being the identity yields that any $A$-module $M$ has a projective resolution of length bounded independently of $M$, c.f.\ \cite[Lemma 3.6]{Lunts:2010}. The argument in loc.\ cit.\ needs a slight modification, though, as there $R$ has global dimension zero whereas here we have finite global dimension.
\end{remark}

\begin{remark}
    It would be interesting to know whether the characterization of Corollary \ref{cor:azumaya_detects_regular} extends to the non-Azumaya classes of the cohomological Brauer group of the separated Noetherian scheme $X$ by looking at the associated twisted derived categories.
    As the Brauer group and cohomological Brauer group coincide when $X$ admits an ample line bundle, so in particular when $X$ is affine, it suffices to have a result showing that strong generation of the perfect complexes for the twisted derived categories can be checked affine locally. (Just as Corollary \ref{introcor:alg_perf_strong_gen_weak_version} shows for Azumaya classes).
    Note that \cite{Toen:2012} has shown such a statement for compact generation, so it would be interesting to see if those methods extend to show strong generation.
\end{remark}

\bibliographystyle{alpha}
\bibliography{mainbib}

\end{document}